\documentclass[a4paper, 11pt]{article}
\usepackage{anysize}
\marginsize{3,5cm}{2,5cm}{2,5cm}{2,5cm}
\usepackage{amssymb}
\usepackage{amsmath,amsbsy,amssymb,amscd}
\usepackage{t1enc}
\usepackage[cp1250]{inputenc}
\usepackage[british]{babel}
\usepackage[all]{xy}
\usepackage{color}
\usepackage{amsfonts}
\usepackage{latexsym}
\usepackage{amsthm}
\usepackage{mathrsfs}
\usepackage{hyperref}
\usepackage{graphicx}

\usepackage{color}
\usepackage{caption}
\usepackage{subcaption}
\usepackage{enumerate}

\definecolor{orange}{rgb}{1,0.5,0}

\usepackage{mathrsfs} 

\DeclareMathAlphabet{\mathpzc}{OT1}{pzc}{L}{it} 

\theoremstyle{definition}
\newtheorem{definition}{Definition}[section]
\newtheorem{theorem}[definition]{Theorem}

\newtheorem{proposition}[definition]{Proposition}
\newtheorem{corollary}[definition]{Corollary}
\newtheorem{lemma}[definition]{Lemma}
\newtheorem{example}[definition]{Example}

\newtheorem{remark}[definition]{Remark}
\newtheorem{question}[definition]{Question}
\newtheorem{claim}[definition]{Claim}

\def\cP{\mathcal{P}}
\def\geq{\geqslant}
\def\leq{\leqslant}
\def\R{\mathbb{R}}
\def\T{\mathbb{T}}

\def\Z{\mathbb{Z}}
\def\N{\mathbb{N}}

\def\id{\mathrm{id}}

\def\cB{\mathbb{B}}

\def\epsilon{\varepsilon}

\def\mf{\mathfrak}
\def\Lie{\operatorname{Lie}}
\def\ad{\operatorname{ad}}
\def\Ad{\operatorname{Ad}}

\def\End{\operatorname{End}}
\def\inj{\operatorname{inj}}

\newcommand{\bea}{\begin{eqnarray}}
  \newcommand{\eea}{\end{eqnarray}}
  \newcommand{\beab}{\begin{eqnarray*}}
  \newcommand{\eeab}{\end{eqnarray*}}

  \newcommand{\be}{\begin{equation}}
  \newcommand{\ee}{\end{equation}}

\newcommand{\set}[1]{\left\lbrace #1 \right\rbrace}
\newcommand{\abs}[1]{\left| #1 \right|}
\newcommand{\mc}{\mathcal}
\newcommand{\norm}[1]{\abs{\abs{#1}}}

\newcommand{\of}{\circ}
\newcommand{\inner}[2]{\left\langle #1, #2 \right\rangle}
\newcommand{\algexp}{\exp_{\operatorname{alg}}}
\newcommand{\geomexp}{\exp_{\operatorname{geom}}}
\newcommand{\alglog}{\log_{\operatorname{alg}}}
\newcommand{\geomlog}{\log_{\operatorname{geom}}}
\newcommand{\mbf}{\mathbf}

\title{Slow Entropy of Some Parabolic Flows}
\author{Adam Kanigowski \and Kurt Vinhage\footnote{K. V. was supported by the National Science Foundation under Award DMS 1604796} \and Daren Wei\footnote{D. W. was partially supported by the NSF grant DMS-16-02409}}
\date{}
\begin{document}
\maketitle
\begin{abstract}
We study nontrivial entropy invariants in the class of parabolic flows on
homogeneous spaces, quasi-unipotent flows. We show that topological complexity (ie, slow entropy) can be computed directly from the Jordan block structure of the adjoint representation. Moreover using uniform polynomial shearing we are able to show that the metric orbit growth (ie, slow entropy) coincides with the topological one, establishing hence variational principle for quasi-unipotent flows (this also applies to the non-compact case). Our results also apply to sequence entropy. We establish criterion for a system to have trivial topological complexity and give some examples in which the measure-theoretic and topological complexities do not coincide for uniquely ergodic systems, violating the intuition of the classical variational principle.
\end{abstract}

\section{Introduction}\label{sec:intro}

The study of dynamical systems typically fits into three paradigms: (partially) hyperbolic, parabolic and elliptic. Quasi-unipotent flows on homogeneous spaces are models for the parabolic regime, exhibiting behaviour such as quantitative equidistribution, polynomial mixing and controlled polynomial divergence of orbits. The most famous example of a quasi-unipotent flow is the horocycle flow on a constant negative curvature surface. Although quasi-unipotent flows have non-trivial statistical properties, including ergodicity, mixing, logarithm laws, they all have zero topological and metric entropy in the usual sense. It is therefore natural to ask for a refinement of the original definition that would allow one to distinguish such flows and describe their complexity. Two methods modifying the standard entropy theory have been developed. The first is the method of A. G. Kushnirenko \cite{Kushnirenko} called {\it sequence entropy} (Section \ref{sec:seq-entropy-def}), in which one allows a growing gaps between ``test times'' to see if orbits have diverged. The second was established by A. Katok and J. P. Thouvenot \cite{Kat-Tho} called {\it slow entropy} (Section \ref{sec:slow-entropy-def}), in which one adapts a dimensional characterization of entropy to characterize orbit growth for systems with lower complexity.

To define sequence entropy, we replace the standard wedge $\bigvee_{n=0}^N T^{-n}\mc P$ with a wedge allowing more time for points to separate, $\bigvee_{n=0}^N T^{-A_n}\mc P$, where $A_n$ is an exponentially growing sequence. One then makes the usual definitions through information functions. This method obtains a family of new isomorphism invariants, which was first used in \cite{Kushnirenko} to distinguish horocycle flow from its square, as well as characterize Kronecker systems as those with zero sequence entropy for any sequence. P. Hulse \cite{Hulse}, D. Newton \cite{Newton1} \cite{Newton2} \cite{Newton3} and E. Kurg \cite{Kurg} further developed sequence entropy and T. N. T. Goodman \cite{Goodman} adapted the definition to the topological setting.

For slow entropy, following the ideas of \cite{Katok2} \cite{Katok3}, one uses a dimensional characterization of entropy by considering coding spaces and coverings by Hamming balls. This was first done in \cite{Kat-Tho}, where a definition for the slow entropy was given in the setting of amenable group actions. Slow entropy was used to give a criterion for the smooth realization of $\Z^k$ actions. Rather than considering a fixed family of times in which to distinguish orbits, this invariant allows one to still consider all times. Instead, one chooses a new ``scale'' for which the asymptotic orbit growth rate is computed. In particular, one may use a polynomial family $(n^\chi)$ instead of an exponential one $(e^{\chi n})$ to obtain a useful invariant for systems with 0 entropy in the usual sense.

We calculate the precise value of slow entropy in polynomial scales any quasi-unipotent flow on finite-volume homogeneous spaces (Theorem \ref{thm:main}). We also compute the sequence entropy of such flows on compact homogeneous spaces (Theorem \ref{thm:mainco}). Furthermore, we also show the variational principle holds for both slow entropy and sequence entropy for quasi-unipotent flows (in the sense that the topological entropies are the supremums of their measure-theoretic entropies over the space of invariant measures). We show that zero topological slow entropy at all scales is equivalent to topological conjugate to a translation of a compact Abelian group, a version of a theorem in \cite{Ferenczi} in the topological category. Finally, we give several non-trivial counter-examples for the variational principle of slow entropy.

These results fit into a program of study for systems with low complexity. M. Ratner \cite{Ratner1} adapted the definition of slow entropy to distinguish horocycle flow and its Cartesian square up to Kakutani equivalence by replacing the Hamming distance with another distance function (i.e. $\bar{f}$-metric). An alternative approach to the slow entropy type invariant was developed by Blume in \cite{Blume1} \cite{Blume2} \cite{Blume3}, where he studied it in setting with slow growth, such as rank one systems. More recent work for slow entropy for higher rank actions appears in A. Katok, S. Katok and F. Rodriguez Hertz \cite{Katok5} and M. Hochman \cite{Hochman}. In the case of flows and transformations, one of the few nontrivial explicit calculations was done  for smooth surface flows by the first author \cite{Kanigowski}, with several applications. The slow entropy of a system also goes by another name, the {\it measure-theoretic complexity}. S. Ferenczi \cite{Ferenczi} proved several results using this terminology, including the characterization of Kronecker systems as those whose measure theoretic complexity is uniformly bounded (ie, has slow entropy 0 at all scales). The work of B. Host, B. Kra, and A. Maass \cite{HoKraMa} computes the topological complexities for nilsystems and establishes some corollaries. Our work generalizes their computation to arbitrary unipotent actions, as well as answer Question 4 of their paper by also computing the measure-theoretic complexity.

\paragraph{Plan of the paper.}
In Sections \ref{sec:slow-entropy-def}-\ref{sec:examples} we give definitions of metric and topological slow entropy, metric and topological sequence entropy, quasi-unipotent flows, introduce some notations and formulate our main results Theorem \ref{thm:main} and Theorem \ref{thm:mainco}. We also give several examples in which our formula can be explicitly computed. In Section \ref{sec:prelims}, we recall fundamental structures and tools from the theory of homogenous spaces. In Section \ref{sec:topologicalSlow}, we calculate the topological slow entropy of quasi-unipotent flows by proving many control lemmas on the decay rates of Bowen balls for comapct homogeneous spaces. In Section \ref{sec:SlowVar}, we calculate the measure-theoretic slow entropy of quasi-unipotent flows on compact homogeneous spaces, as well as discuss the relationships between the metric and topological slow entropies (in particular, we show that the variational principle does not hold for general slow entropy results). In both Sections \ref{sec:topologicalSlow} and \ref{sec:SlowVar}, we treat only compact homoeneous spaces because they are much easier to handle. In Section \ref{sec:noncompact}, we extend the results on slow entropy to the case of noncompact homogeneous spaces. We remark that this is more difficult since even though the divergence rates are the same as their compact counterparts, a local analysis is not sufficient because the injectivity radius can tend to 0. In Section \ref{sec:SeqeunceEntropy}, we calculate the both topological and metric sequence entropy of quasi-unipotent flow. In Section \ref{sec:specific}, we show the details in the computations for the examples in Section \ref{sec:examples}.

\subsection{Definition of Entropy Invariants Considered}
\label{sec:slow-entropy-def}

Because we consider many entropy-type invariants, we gather the definitions here for reference throughout the paper.

\subsubsection{Topological Slow Entropy}

Let $\varphi^t : X \to X$ or $f : X \to X$ be a uniformly continuous flow or transformation of a locally compact metric space. The following definitions are for flows, the analogous definitions for transformations are easily deduced. The {\it Bowen ball} of radius $\epsilon$ up to time $T$ (called a $(\epsilon,T)$-Bowen ball for short) is defined as:

\[ B_\varphi^T(x,\epsilon) = \bigcap_{t \in [0,T]} \varphi^{-t}(B(\varphi^t(x),\epsilon)) \]

If $K \subset X$ is a compact subset of $X$, let $N_{\varphi,K}(\epsilon,T)$ be the minimal number of $(\epsilon,T)$-Bowen balls required to cover $K$ (since $K$ is compact, this is finite). Let $S_{\varphi,K}(\epsilon,T)$ be the maximal number of $(\epsilon,T)$-Bowen balls which can be place disjointly in $X$ with centers in $K$. If $a : \R_+ \to \R_+$ is a function such that $a(T) \to \infty$ as $T \to \infty$, consider the asymptotics:

\begin{equation}
\label{eq:topo-limit}
\delta_{\varphi,K,a}^N = \limsup_{T \to \infty} \dfrac{N_{\varphi,K}(\epsilon,T)}{a(T)} \qquad \delta^S_{\varphi,K,a} =  \limsup_{T \to \infty} \dfrac{S_{\varphi,K}(\epsilon,T)}{a(T)}
\end{equation}

 We now assume that we have a fixed family of functions $a_\chi : \R_+ \to \R_+$, $\chi \in \R_+$ such that if $\chi < \chi'$, then $a_\chi = o(a_{\chi'})$.

\begin{definition}
The {\it (slow) topological entropy of $\varphi$ with respect to $\set{a_\chi}$} is:

\[ h_{\operatorname{top},a_\chi}(\varphi) = \sup_K \lim_{\epsilon \to 0} \sup \set{ \chi : \delta_{\varphi,K,a_\chi}^N > 0} = \sup_K \lim_{\epsilon \to 0} \sup \set{ \chi : \delta_{\varphi,K,a_\chi}^S > 0}  \]

There are two particularly useful class of functions $a_\chi$. One is the exponential class $a_\chi(T) = e^{\chi T}$. In this case the entropy defined above is the usual topological entropy. The other is the polynomial class $a_\chi(T) = T^\chi$, in which case we call $R$ the {\it polynomial slow entropy of $\varphi$} and say that the polynomial slow entropy of $\varphi$ {\it gives rate $T^R$}.
\end{definition}

The topological slow entropy is essentially the same as the topological complexity considered in \cite{HoKraMa}, however the topological complexity is an equivalence class of rates, whereas the topological slow entropy simplifies the notion by attaching a single number to the topological complexity. One easily deduces the topological complexity from our arguments.

\subsubsection{Metric Slow Entropy}
Our definition will match the definition of slow entropy following \cite{Kat-Tho} and \cite{Kanigowski}. While not exactly, we note that this type of invariant is closely related to the measure-theoretic complexity used in \cite{Ferenczi}. The same remarks made about the difference between topological complexity and topological slow entropy also hold here.

We make definitions for flows, analogous definitions for transformations can be deduced easily.
Let $(T_t)_{t\in\R}$ act on $(X,\cB,\mu)$ and
$\cP=\{P_1,...,P_k\}$ be a finite measurable partition of $X$. For $T\in \R_+$ and $x\in X$ we define the {\it coding} of $x$ to be a function $\phi_{\cP,r}(x) : [0,T] \to \set{1,\dots,k}$ defined by
\begin{equation}\label{eq:par0}
 \phi_{\cP,r}(x)=(x_s)_{s\in[0,T]},\text{ where }   x_s=i\; \text{   whenever  }\; T_{s}x\in P_i.
\end{equation}
For any two points $x,y\in X$ their {\it Hamming distance} with respect to $\mc P$ is given by
$$
\bar{d\,}_{\varphi,\mc P}^T(x,y):=\frac{T-\left|\{s\in[0,T]\;:\; x_s=y_s\}\right|}{T}.
$$
For $\epsilon>0$, let $B_{\varphi,\mc P}^T(x,\epsilon) = \set{ y \in X : \bar{d\,}_{\varphi,\mc P}^T(x,y) < \epsilon}$ be the {\it $\epsilon$-Hamming ball centered at $x$}. Then as in the definition of topological slow entropy, define:
\begin{equation}\label{bal1}
S(\cP,T,\epsilon) = \min \set{\abs{F} : \mu\left(\bigcup_{x\in F} B_{\varphi,\mc P}^T(x,\epsilon)\right) > 1-\epsilon}
\end{equation}
Analogously to \eqref{eq:topo-limit} for a function $a:\R_+\to \R_+$ such that $a(T)\to \infty$ as $T\to \infty$ we define
\begin{equation}\label{eq:metric-limit}
A(\varphi,\cP,\epsilon,a) = \limsup_{T \to \infty} \dfrac{S(\cP,T,\epsilon)}{a(T)}
\end{equation}
 We now assume that we have a fixed family of functions $a_\chi : \R_+ \to \R_+$, $\chi \in \R_+$ such that if $\chi < \chi'$, then $a_\chi = o(a_{\chi'})$.
\begin{definition}\label{def:metric}
We say that the {\it (slow) metric entropy of $\varphi$ with respect to $\set{a_\chi}$ is $h_{m,a_\chi}(\varphi)$} if:
$$
h_{m,a_\chi}(\varphi) =\sup_{\cP-\text{ finite partition}} h_{m,a_\chi}(\varphi,\cP),$$
where
$$
h_{m,a_\chi}(\varphi,\cP):=\lim_{\epsilon \to 0}\left( \sup \{\chi: A(\varphi,\cP,\epsilon,a_\chi)>0\}\right).
$$
\end{definition}

Recall that a (finite) partition $\cP$ is called a {\em generator} if the minimal $(T_t)_{t\in\R}$ invariant $\sigma$- algebra containing $\cP$ is the whole $\sigma$-algebra $\cB$.
\begin{proposition}(\cite{Kat-Tho}, Proposition 1.)\label{prop:gen} If $\cP$ is a generator, then
$$
h_{m,a_\chi}(\varphi)=h_{m,a_\chi}(\varphi,\cP).
$$
\end{proposition}

\subsubsection{Measure-Theoretic Sequence Entropy}
\label{sec:seq-entropy-def}

Recall if $\mc P = \set{P_1,\dots,P_k}$ is a partition of a measure space $(X,\mu)$, the entropy of $\mc P$ is defined as $H(\mc P) = \sum_{i=1}^k \mu(P_i) \log \mu(P_i)$.

\begin{definition}
Suppose that $T : (X,\mu) \to (X,\mu)$ is an invertible measure preserving transformation or flow on a Lebesgue space $X$ with probability measure $\mu$. For a given increasing sequence $A=\{t_1,t_2,\ldots,t_n,\ldots\}$ with $t_i \to \infty$, a measurable partition (finite or countable partition) $\mc P$ such that $H(\mc P)<\infty$, define
\begin{equation}
\begin{aligned}
&h_A(T,\xi)=\limsup_{n\to\infty}\frac{1}{n}H(\bigvee_{i=1}^nT^{-t_i}\mc P)\\
&h_{A,\mu}(T)=\sup_{\mc P}h_A(T,\mc P).
\end{aligned}
\end{equation}

We call $h_{A,\mu}(T)$ the {\it measure-theoretic sequence entropy of $T$ with respect to $A$}.
\end{definition}

\subsubsection{Topological Sequence Entropy}
\label{sec:top-seq-entropy}

\begin{definition}
Suppose that $(X,d)$ is a locally compact metric space and $T:X\to X$ is a flow or transformation. Let $\set{t_n : k \in \N_0}$ be any increasing sequence of real numbers (natural numbers in the case of a transformation) tending to $\infty$. Then define $B_T^N(x,\epsilon;\set{t_n}) = \bigcap_{n=0}^\infty T^{-t_n}(B(T^{t_n}(x),\epsilon))$. We say a set $E\subset X$ is $(A,n,\epsilon)-$separated (with respect to $T$) if $B_\varphi^N(x,\epsilon;\set{t_n}) \cap B_T^N(y,\epsilon;\set{t_n}) = \emptyset$ for every $x \not= y \in E$. For a compact set $K\subset X$, let $N(A,n,\epsilon,K)$ denote the largest cardinality of any $(A,n,\epsilon)-$separated set in $K$. Then define
\begin{equation}
\begin{aligned}
&h(A,\epsilon,K)=\limsup_{n\to\infty}\frac{1}{n}\log N(A,n,\epsilon,K)\\
&h_A(T)=\sup_K\sup_{\epsilon > 0}N(A,\epsilon,K)
\end{aligned}
\end{equation}

$h_A(T)$ is called {\it the topological sequence entropy of $T$ with respect to $A$}.
\end{definition}

\subsection{Statement of the Main Theorem}
\label{sec:results}

Let $\mf g$ be a Lie algebra of a connected Lie group $G$. An element $U \in \mf g$ is called {\it quasi-unipotent} if it can be written as $U = U' + Q$, where $\ad_{U'}^N = 0$ for some $N$, $Q$ is $\ad$-compact, and $[Q,U'] =0$. See Section \ref{sec:prelims} for any further undefined terms.

\begin{definition}
	\label{def:chain-basis}	
	Let $\mf g$ be a Lie algebra and $U = U' + Q \in \mf g$ be a quasi-unipotent element. A {\it chain in $\mf g$ with respect to $U$ of depth $m$} is a linearly independent set $\set{ X_j : 0 \le j \le m}$  such that $X_0$ is in the centralizer of $U$ and:
	
	\[ \ad_U(X_j) = X_{j-1} \mbox{ for all }1 \le j \le m\]

 A {\it double chain in $\mf g$ with respect to $U$ of depth $m$} is a linearly independent set \\ $\set{X_{j,i} : 0 \le j \le m, i = 0,1}$ and a number $\alpha$ such that $\ad_Q(X_{j,0}) = -\alpha X_{j,1}$, $\ad_Q(X_{j,1}) = \alpha X_{j,0}$ for all $0 \le j \le m$, $X_{0,i}$ is in the centralizer of $U'$ for $i = 0,1$ and:

\[ \ad_{U'}(X_{j,i}) = X_{j-1,i} \mbox{ for all }1 \le j \le m, i = 0,1\]

	A chain basis of $\mf g$ with respect to $U$ is a basis of chains and double chains. Then sequence of depths $(m_1,\dots,m_n)$ of the chains and double chains, with each double chain listed twice, is called the {\it chain structure} of $U$.

\end{definition}

We will see that all quasi-unipotent elements have a chain basis (Lemma \ref{lem:chain-basis}).  A quasi-unipotent element $U \in \Lie(G)$ induces a quasi-unipotent flow on a homogeneous space $\Gamma \backslash G$ by:

\[ \varphi_t(\Gamma g) = \Gamma g \algexp(tU)\]

where $\Gamma$ is a lattice (cofinite volume, discrete subgroup) and $G$ is unimodular.

\begin{theorem}
\label{thm:main}
Let $a_\chi(t) = t^\chi$ and $\varphi_t$ be a quasi-unipotent flow on a finite volume homogeneous space with chain structure $(m_1,\dots,m_n)$. Then the topological and measure-theoretic slow entropy of $\varphi_t$ with respect to $a_\chi$ are both $R$, where:

\[ R = \sum_{i=1}^n m_i(m_i+1)/2\]
\end{theorem}


As a corollary to Theorem \ref{thm:main}, one obtains a version of the variational principle for quasi-unipotent flows and slow entropy. Namely, that the topological slow entropy of a quasiunipotent flow is the supremum of the measure-theoretic slow entropies taken over all invariant measures. We will see that this property is fairly special (Section \ref{sec: counterexampleVar}).

\begin{theorem}\label{thm:mainco}
Let $\varphi_t$ be a quasi-unipotent flow on a comapct homogeneous space. Then with respect to the sequence $A_n = C\lambda^n$, the measure-theoretic sequence entropy and topological sequence entropy of $\varphi_t$ is $R \log \lambda$.
\end{theorem}

\begin{remark}
The compactness assumption allows one to guarantee that any separation of orbits happens at a local level: separation for the flow implies the points also separate in the universal cover. Because the injectivity radius becomes arbitrarily small, one needs estimates on how often points enter a cusp, separate on the universal cover, but remain close in $\Gamma \backslash G$. For Section \ref{sec:prelims} to Section \ref{sec:SlowVar} we concentrate on cocompact lattice case and extend our result to non-compact lattice in Section \ref{sec:noncompact} for slow entropy.
\end{remark}

\subsection{Examples of Quasi-Unipotent Flows}
\label{sec:examples}

The three main classes of examples of quasi-unipotent flows come from the case when:

\begin{enumerate}[(i)]
	\item $G$ is a semisimple group
	\item $G$ is a nilpotent group
	\item $G = H \ltimes N$ is the semidirect product of a semisimple group $H$ with a nilpotent group $N$
\end{enumerate}

We give an example of flows for each type described above.

\begin{example}
\label{ex:semisimple}
	Let $G = SL(d,\R)$, $\Gamma \subset G$ be any lattice, and
	
	\[ U = \begin{pmatrix}
	0 & 1 \\
	& 0 &  1 \\
	& & \ddots & \ddots \\
	& & & 0 & 1 \\
	& & & & 0
	\end{pmatrix} \qquad \algexp(tU) = \begin{pmatrix}
	1 & t & t^2/2 & \cdots & t^{d-1}/(d-1)! \\
	& 1 & t & \cdots & t^{d-2}/(d-2)! \\
	& & \ddots & \ddots & \vdots \\
	& & & 1 & t \\
	& & & & 1
	\end{pmatrix}\]
	
	We call $U$ the {\it principal} nilpotent element associated to the algebra $\mf{sl}(d,\R)$. In the special case of $SL(d,\R)$, any nilpotent algebra element is conjugate to a block form element:
	
	\[ U = \begin{pmatrix}
	U_1 \\
	& U_2 \\
	& & \ddots \\
	& & & U_n
	\end{pmatrix}
	\]
	
	where each $U_i \in \mf{sl}(d_i,\R)$ is the principal element. Note that this is exactly the Jordan normal form of the matrix $U$. Call each $U_i$ a {\it principal block} and the sequence $(\dim(U_1),\dots,\dim(U_n))$ the {\it block sequence of $U$}.
\end{example}

\begin{example}
\label{ex:nilpotent}
	Let $N$ be the matrix Lie group with Lie algebra:
	
	\[ \mf n = \set{ U(\mbf x,t) : \mbf x \in \R^d,t\in \R} \qquad U(\mbf x,t) = \begin{pmatrix}
	0 & x_1 & x_2 &  x_3&  \dots & x_d \\
	& 0& t & -t/2  & \dots & (-1)^dt/d \\
	& & 0& t  & \dots & (-1)^{d-1}t/(d-1) \\
	& & & \ddots& \ddots  & \vdots \\
	& & & & 0& t \\
	& & & & & 0
	\end{pmatrix} \]
	
	Let $\Gamma = N(\Z)$ be the $\Z$-points of $N$ with the coordinates described. One may check that $\Gamma$ is a lattice since it contains the $\Z^d$ coming from the $x_i$ coordinates, and that

\[ \algexp(U(\mbf 0,1)) = \begin{pmatrix}1 & 0 & 0 & \cdots & \cdots & 0 \\
                                                         & 1 & 1  & \\
                                                         &  & 1 & 1 \\
                                                           &   &   & \ddots & \ddots \\
                                                           & & &            & 1 & 1 \\
                                                             & & &       & & 1
 \end{pmatrix} \in \Gamma,\]

The unipotent flow induced by the algebra element \[ U = U((\alpha,-\alpha/2,\dots,(-1)^{d+1}\alpha/(d+1)),1)\] induces a unipotent flow on $\Gamma \backslash N$ which is transverse to the torus corresponding to the $x$-coordinates. The first return map is exactly the classical affine transformation:
	
	\[ F(x_1,\dots,x_d) = (x_1+\alpha,x_2+x_1,\dots,x_d+x_{d-1})\]
\end{example}

\begin{example}
\label{ex:twisted}
	Let $\rho : SL(d,\R) \to SL(N,\R)$ be a rational representation of $SL(d,\R)$.  Then let $G$ be the group:
	
	\[ \begin{array}{c}
	G = \set{ \begin{pmatrix}
		\rho(A) & v \\
		0 & 1
		\end{pmatrix} : A \in SL(d,\R), v \in \R^N} \\
	\\
	\Gamma = \set{ \begin{pmatrix}
		\rho(A) & v \\
		0 & 1
		\end{pmatrix} : A \in SL(N,\Z) \cap \rho(SL(d,\R)), v \in \Z^N}
	\end{array}
	\]
	
		Then take any nilpotent element of $\mf{sl}(d,\R)$ gives a flow on $\Gamma \backslash G$. This unipotent flow is a twisted combination of the two previous ones. Indeed, if we factor onto the semisimple component, we get Example \ref{ex:semisimple}. Furthermore, if $\gamma \in SL(N,\Z)\cap \rho(SL(d,\R))$ corresponds to a closed orbit of $U$ in the semisimple factor, then taking the restriction to that closed orbit gives a system similar to Example \ref{ex:nilpotent}.
\end{example}

We get the following Corollary of Theorem \ref{thm:main} (proofs can be found in section \ref{sec:specific}):

\begin{corollary}
\label{cor:specific}
	The topological and metric slow entropy of Examples \ref{ex:semisimple}, \ref{ex:nilpotent} and \ref{ex:twisted} give rate $T^R$, where:
	
	\begin{enumerate}[(i)]
		\item $R=\sum_{i=1}^{m}\frac{1}{6}k_i(4k_i+1)(k_i-1)+\sum_{i=1}^{m-1}\sum_{j=i+1}^m\frac{1}{3} k_i \left(k_i^2+3k_j^2-3k_j-1\right) $, where $(k_1,\dots,k_m)$, $k_i \le k_{i+1}$ is the block sequence of $U$\label{sld-formula}
		\item $R = d(d-1)/2$
		\item $R = R_{ss} + R_\rho$, where $R_{ss}$ is the rate coming from Part (\ref{sld-formula}) and $R_\rho = \sum \ell_i(\ell_i-1)/2$. Here, $\ell_i$ are the lengths of the Jordan blocks of $d\rho(U)$ acting on $\R^N$
	\end{enumerate}
\end{corollary}

Finally, we give one last way of practically calculating the slow entropy of $U$ in the case of a semisimple group. If $U = U' + Q \in \mf g$ is a quasi-unipotent elements of a semisimple Lie algebra, an $\mf{sl}(2,\R)$-triple is a triple $(V,X,U')$ such that:

\[ [X,U'] = 2U' \qquad [X,V] = -2V \qquad [U',V] = X \]

The well-known Jacobson-Morosov theorem ensures the existence of $\mf{sl}(2,\R)$-triples for arbitrary unipotent elements. Let $G$ be semisimple, $U' \in \mf g$ be unipotent, $C(U')$ be the centralizer of $U'$ in $\mf g$ and $(V,X,U')$ be an $\mf{sl}(2,\R)$-triple associated to $U'$. Then $\ad_X(C(U')) = C(U')$, and all eigenvalues of $\ad_X$ on $C(U') \subset \mf g$ are non-negative integers. Let $d_n$ denote the dimension of the eigenspace for the eigenvalue $n$.

\begin{corollary}
\label{cor:sl2-triple}
The topological and metric slow entropy of the quasi-unipotent flow induced by $U$ with respect to the polynomial family is $R$, where:

\[ R = \sum_{n=0}^\infty d_n \cdot n(n+1)/2\]
\end{corollary}

\section{Preliminaries on Homogeneous Spaces}
\label{sec:prelims}

In this section, we recall some basic facts from the theory of Lie groups and homogeneous spaces. Let $G$ be a Lie group with Lie algebra $\mf g$. Given $g \in G$, let $L_g, R_g : G \to G$ denote the left and right translations on $G$. Let $\algexp : \mf g \to G$ denote the exponential mapping of the Lie algebra $\mf g$ onto $G$. Then $\algexp$ has a local inverse $\alglog$ sending a neighborhood of $e \in G$ to a neighborhood of $0 \in \mf g$.

\subsection{Metrics and Measures on Homogeneous Spaces}
\label{sec:metrics}

Let $\Gamma \subset G$ be a (discrete) subgroup. We introduce a metric on a the homogeneous space $\Gamma \backslash G$ by first introducing a left-invariant metric on $G$. Fix an inner product $\inner{\cdot}{\cdot}_0$ on $\mf g$, and define for $v,w \in T_gG$:

\[ \inner{v}{w} = \inner{dL_{g^{-1}}v}{dL_{g^{-1}}w}_0\]

By construction, $\inner{\cdot}{\cdot}$ is left-invariant, so it induces a Riemannian metric on the space $\Gamma \backslash G$. The Riemannian metric also has an associated exponential mapping $\geomexp : \mf g \to G$, which is $C^\infty$ and satisfies

\begin{equation}
d_0\geomexp = \id.
\end{equation}

Like the algebraic exponential, there is a local inverse of $\geomexp$ which we will denote by $\geomlog$. The following is immediate from the definition of the inner product.

\begin{lemma}
The Riemannian volume is a (left) Haar measure on $G$. In particular, it is independent of the metric $\inner{\cdot}{\cdot}_0$ when determining a probability measure on a homogeneous space.
\end{lemma}

We say that $\bar{x} \in \Gamma \backslash G$ is {\it $\epsilon$-good} if the map $f_{\bar{x}} : B(0,\epsilon) \to \Gamma \backslash G$ is injective, where $f_{\bar{x}}(X) = \bar{x}\geomexp(X)$. Let $\mu$ denote the Haar measure of $\Gamma \backslash G$. By cocompactness of $G$, there exists $\epsilon_0$ such that every $\bar{x} \in \Gamma \backslash G$ is $\epsilon$-good whenever $\epsilon < \epsilon_0$.


		

\subsection{The Adjoint Representation}

 $G$ acts on itself by conjugation $C_g : h \mapsto g^{-1}hg$, and taking the derivative at the identity in the coordinate $h$ gives the {\it adjoint representation} of $G$ on $\mf g = T_eG$, $\Ad : G \to GL(\mf g)$. Taking the derivative of this map in the $g$ coordinate yields the Adjoint representation of the Lie algebra $\mf g$, $\ad : \mf g \to \End(\mf g)$, which coincides with the Lie bracket: $\ad(X)Y = [X,Y]$. The following are standard tools from the theory of Lie groups, which we write as a Lemma to reference.

\begin{lemma}
	\label{lem:adjoint}
	If $X,Y \in \mf g$,
	
	\[ \algexp(-X)\algexp(Y)\algexp(X)= \algexp(\Ad(X)Y) \]
	
	\[ \exp(\ad(X)) := \sum_{k=0}^\infty \dfrac{\ad(X)^k}{k!}  = \Ad(\algexp(X))\]
\end{lemma}

\section{Topological Slow Entropy of Quasi-Unipotent Flows}\label{sec:topologicalSlow}

In this section we will prove the existence chain bases and several control lemmas on the decay rates of Bowen balls to calculate the topological slow entropy.

\subsection{Chain Bases}

Recall Definition \ref{def:chain-basis}.

\begin{lemma}\label{lem:chain-basis}
If $\mf g$ is a Lie algebra, every quasi-unipotent element $U \in \mf g$ has a chain basis
\end{lemma}

\begin{proof}
	By definition, an element is quasi-unipotent if and only if $\ad_U$ has purely imaginary eigenvalues. The Jordan normal form of $\ad_U$ allow us to find a basis of blocks, where the real eigenvalues (only 0) are upper triangular blocks and the imaginary eigenvalues have block-upper triangular structure. For each block of the form:
	
	\[\begin{pmatrix}
	0 & 1 \\
	  & 0 & 1 \\
	 & & \ddots &\ddots \\
	 & & & 0 & 1 \\
	 & & & & 0
	\end{pmatrix}\]
	
	let $X_{j}^i$ be elements of $\mf g$ corresponding the the $e_{j+1}$. By definition, $X_0^i$ is the eigenvector, and $\ad_U(X_j^i) = X_{j-1}^i$. Let $Q_\alpha = \begin{pmatrix} 0 & \alpha \\ -\alpha & 0 \end{pmatrix}$ and $I_2$ be the identity $2 \times 2$ matrix, so that the eigenvalues of $Q_\alpha$ are $\pm \alpha i$. For each block of the form:

        \[\begin{pmatrix}
         Q_\alpha & I_2 \\
         & Q_\alpha & I_2 \\
         & & \ddots & \ddots \\
         & & & Q_\alpha & I_2 \\
         & & & & Q_\alpha
           \end{pmatrix} \]

let $X_{j,0}^i$ denote the element corresponding to $e_{2j+1}$ and $X_{j,1}^i$ denote the element corresponding to $e_{2j+2}$.
\end{proof}

Recall that in Section \ref{sec:metrics} we were able to choose the metric $\inner{\cdot}{\cdot}_0$ on $\mf g$ in any way we wish. We now specify it by fixing a chain basis of $\mf g$ for $U$, and declaring it to be orthonormal (this is well-defined and determines the metric uniquely).

\subsection{Approximation of Bowen Balls}

In this section, we work in greater generality by letting $\Gamma$ be any discrete subgroup (not necessarily cocompact or cofinite volume). In particular, we allow the case of $\Gamma = \set{e}$. Recall that the Bowen metrics are defined as $d_T^\varphi(\bar{x},\bar{y}) = \sup_{0 \le t \le T} d(\varphi^t(\bar{x}),\varphi^t(\bar{y}))$. In particular, the $(\epsilon,T)$-Bowen ball is contained in the usual ball of radius $\epsilon$. Given $\bar{x},\bar{y} \in \Gamma \backslash G$, we may write $d_{\Gamma \backslash G}(\bar{x},\bar{y}) = d_G(x,y)$ for some pair of points $x,y \in G$, by definition. Furthermore, we may write $y = xg$ for some $g \in G$, in which case $d_G(x,y) = d_G(x,xg) = d_G(e,g)$ since our metric is left-invariant. Say that $\bar{x} \in \Gamma \backslash G$ is {\it $\epsilon$-good} if the map $f_{\bar{x}} : B(0,\epsilon) \to \Gamma \backslash G$ is injective, where $f_{\bar{x}}(X) = \bar{x}\geomexp(X)$. Then if $K$ is a compact set in $\Gamma \backslash G$, let $\inj(K) = \sup \set{ \epsilon \ge 0 : \mbox{every }\bar{x} \in K\mbox{ is }\epsilon\mbox{-good}}$. If $K = \set{\bar{x}}$, we will denote $\inj(\set{\bar{x}}) = \inj(\bar{x})$, and call $\inj(K)$ the {\it injectivity radius} of $K$. Let $\mu$ denote the Haar measure of $\Gamma \backslash G$.

\begin{lemma}\label{lem:CompactApprox}
For every $\delta > 0$, there a compact subset $K \subset \Gamma\backslash G$ such that $\mu(K) > 1-\delta$. Furthermore, for any compact set $K$, $\inj(K) > 0$.
\end{lemma}
\begin{proof}
Note that $\Gamma \backslash G$ is an increasing union of compact sets $X_r = \geomexp(B(0,r)) \subset \Gamma \backslash G$. We may thus find $r$ such that $\mu(X_r) > 1- \delta$. Since $\bar{x} \mapsto \inj(\bar{x})$ is continuous and positive, for any compact $K$, there exists some $\epsilon$ such that $\inj(K) > \epsilon$.
\end{proof}

Note that if $\Gamma$ is cocompact or $\Gamma = \set{e}$, then $\inj(\Gamma \backslash G) > 0$ (in fact for lattices, the converse is true as well). Note that

\begin{equation}
\label{eq:homo-divergence}
d_{\Gamma \backslash G}(\varphi_t(\bar{x}),\varphi_t(\bar{y})) = d_G(x\algexp(tU),xg\algexp(tU)) = d_G(e,\algexp(-tU)g\algexp(tU))
\end{equation}

as long as $d_G(e,\algexp(-tU)g\algexp(tU)) < \inj(\varphi_t(\bar{x}))$. Note that this only holds for a bounded (but large) set of $t \in \R$. However, because we wish to consider the case when $\sup_{0 \le t \le T} d(\varphi^t(\bar{x}),\varphi^t(\bar{y})) < \epsilon$, in the case of a cocompact lattice we are free to apply \eqref{eq:homo-divergence}.

If $\epsilon$ is sufficiently small, then we can write $g = \algexp(X)$ for some $X \in \mf g$.  Fix a quasi-unipotent element $U = U' + Q$. We may therefore write $X$ in terms of the fixed chain basis chosen, $X = \sum a_j^i(0)X_{j}^i + \sum b_j^i(0)X_{j,0}^i + c_j^i(0)X_{j,1}^i$, where the first sum corresponds to the chains and the second to the double chains. Then let $a_j^i(t)$, $b_j^i(t)$ and $c_j^i(t)$ be the functions depending on $t$ by the formula:

\begin{equation}
\label{eq:define-Xt}
X_t := \alglog(\algexp(-tU)\algexp(X)\algexp(tU)) = \sum_{i,j} a_j^i(t)X_{j}^i + \sum_{i,j} b_j^i(t)X_{j,0}^i + c_j^i(t)X_{j,1}^i
\end{equation}

\begin{lemma}
\label{lem:polynomial-growth}
The functions $a_j^i(t)$ are polynomials in $t$, whose coefficients depend linearly on $\set{a_j^i(0)}$. Furthermore, the coefficients of $t^k$ for the polynomials $a_0^i(t)$ are proportional to $a_k^0(0)$, with the proportionality given by universal constants.

The functions $b_j^i(t)$ and $c_j^i$ are given by

\begin{eqnarray*}
b_j^i(t) & = & (\cos(\alpha t)f_j^i(t) - \sin(\alpha t)g_j^i(t))  \\
c_j^i(t) & = & (\sin(\alpha t) f_j^i(t)+ \cos(\alpha t)g_j^i(t))
\end{eqnarray*}

 for polynomials $f_j^i(t)$ and $g_j^i(t)$, whose coefficients depend lienarly on $\set{b_j^i(0)}$ and $\set{c_j^i(0)}$, respectively.  Furthermore, the coefficients of $t^k$ for the polynomials $f_0^i(t)$ and $g_0^i(t)$ are proportional to $b_k^i(0)$ and $c_k^i(0)$, respectively, with the proportionality given by universal constants.
\end{lemma}

\begin{proof}
	Since all definitions are linear, and each chain and double chain is preserved by $\ad_U$, we may address each chain and double chain separately. We first address the case of chains. Assume that $X = \sum_j a_j(0)X_j$ lies in a single chain of depth $m$. Using the relations determined in the definition of chain basis (Definition \ref{def:chain-basis}), we get:
	
\begin{eqnarray*}
	X_t & = & \alglog(\algexp(-tU)\algexp(X)\algexp(tU)) \\
	    & = & \alglog\algexp(\Ad(\algexp(tU))X) \\
	    & = & \Ad(\algexp(tU))X \\
	    & = & \exp(\ad(tU))X \\
	    & = & \sum_{k=0}^\infty \ad(tU)^kX/k! \\
	    & = & \sum_{j=0}^m \sum_{k=0}^{j} a_j(0)\ad(tU)^kX_j/k! \\
	    & = & \sum_{j=0}^m \sum_{k=0}^{j} t^k a_j(0) X_{j-k}/k! \\
	    &= & \sum_{j=0}^m \left(\sum_{k=j}^{m} \dfrac{t^{k-j}a_{k}(0)}{(k-j)!} \right) X_j
\end{eqnarray*}
	
   For the last part of the lemma, notice that the polynomial for the $X_{0}$ term is exactly $a_{0}(t) = \sum_{k=0}^{m} t^k a_{k}(0) /k!$. Now suppose that $\set{X_{j,0},X_{j,1}}$ is a double chain. Then note that since $U = U' + Q_\alpha$ and $[U',Q_\alpha] = 0$, $\Ad(\exp(tU)) = \Ad(\exp(tQ_\alpha)\exp(tU')) = R_{t\alpha}\Ad(\exp(tU'))$, where $R_{t\alpha} = \exp(tQ_\alpha)$ is the rotation matrix in each paired level of the double chain. Then we may proceed as above, considering two chains for $U'$ independently, pulling $R_{t\alpha}$ out of the sum to compute. This gives the result.

\end{proof}

We can understand the rates of divergence using the following fact about polynomials:

\begin{lemma}
	\label{lem:polynomial-coefficients}
	Let $p(t) = \sum_{k=0}^d a_kt^k$ be a polynomial of degree $d$. There exists $C(d)$ such that if $\abs{p(t)} < \epsilon$ for all $t \in [0,T]$, then $\abs{a_k} < C(d)T^{-k}\epsilon$ for all $k = 0,\dots,d$. Conversely, if $\abs{a_k} < C(d)^{-1}T^{-k}\epsilon$ for all $k$, then $\abs{p(t)} < \epsilon$ for all $t \in [0,T]$.
\end{lemma}

\begin{proof}
	Let $\mc P_d$ denote the space of polynomials of degree at most $d$. Then $\mc P_d$ carries a family of norms $\norm{p}_T = \sup_{t \in [0,T]} \abs{p(t)}$, as well as the norm $\norm{p}_\infty = \max_k \abs{a_k}$, where $p(t) = \sum a_kt^k$. Since $\mc P_d \cong \R^d$, and all norms on $\R^d$ are equivalent, we conclude that there exists $C(d)$ such that $C(d)^{-1} \norm{p}_\infty \le \norm{p}_1 \le C(d) \norm{p}_\infty$.
	
	Now we compute:
	
	\begin{eqnarray*}
	\sup_{0 \le t \le T} \abs{\sum_{k=0}^d a_kt^k} & = & \sup_{0 \le t \le 1} \abs{\sum_{k=0}^d (a_kT^k)t^k}\\
	& \ge & C(d)^{-1} \max_k \abs{a_kT^k} \\
	\abs{a_k} & \le & C(d)T^{-k}\norm{p}_T
	\end{eqnarray*}
	
	The reverse computation follows similarly.
\end{proof}

\begin{corollary}
    \label{cor:twisted-coefficients}
	Let $p(t) = \cos(\alpha t)f(t) - \sin(\alpha t)g(t)$ and $q(t) = \sin(\alpha t) f(t) + \cos(\alpha t) g(t)$ with $f,g$ polynomials of degree $d$ and coefficients $f_k$ and $g_k$, respectively. Then there exists $E(d)$ such that if $p(t)^2 + q(t)^2 < \epsilon^2$ for all $t \in [0,T]$, then $\abs{f_k},\abs{g_k} < E(d) T^{-k}\epsilon$ for all $k = 0,\dots,d$. Conversely if $\abs{f_k},\abs{g_k} < E(d)^{-1} T^{-k}\epsilon$ for all $k$, then $p(t)^2 + q(t)^2 < \epsilon^2$ for all $t \in [0,T]$.
\end{corollary}

\begin{proof}
The proof is almost identical to that of Lemma \ref{lem:polynomial-coefficients}. We let $\mc Q_d$ denote the space of pairs of functions $p(t)$ and $q(t)$ as defined above, which is isomorphic to $\R^{2d}$, where the coordinates are given by the coefficients of $f$ and $g$. We can again put a norm $\norm{\cdot}_\infty$ corresponding to the maximum coefficients and the max norm on the interval $[0,T]$, $\norm{\cdot}_T$ for the function $\sqrt{p(t)^2 + q(t)^2} = \sqrt{f(t)^2+g(t)^2}$. Then if $E(d)$ is the constant which makes the norms equivalent:

\begin{eqnarray*}
        \sup_{0 \le t \le T} \sqrt{\left(\sum_{k=0}^d f_kt^k\right)^2 + \left(\sum_{k=0}^d g_kt^k\right)^2} & = & \sup_{0 \le t \le 1} \sqrt{\left(\sum_{k=0}^d (f_kT^k)t^k\right)^2 + \left(\sum_{k=0}^d (g_kT^k)t^k\right)^2}\\
	& \ge & E(d)^{-1} \max  \set{\abs{f_kT^k},\abs{g_kT^k}} \\
	\abs{f_k},\abs{g_k} & \le & E(d)T^{-k}\norm{(p,q)}_T
\end{eqnarray*}

Again, the inequalities are easily reversible to obtain the converse.
\end{proof}

Introduce the following norms on $\mf g$: $\norm{\cdot}$, the usual Reimannian norm,  and $$\norm{\sum c_j^iX_j^i + \sum c_{j,0}^iX_{j,0}^i + c_{j,1}^i X_{j,1}^i}_\infty = \max \set{\abs{c_j^i},\sqrt{(c_{j,0}^i)^2 + (c_{j,1}^i)^2}}.$$ By definition of $\geomexp$, $d_G(e,\geomexp(X)) = \norm{X}$.

\begin{lemma}
	\label{lem:bilipischitz}
	If $\epsilon < \inj(\varphi^t(\bar{x}))$ for all $t \in [a,b]$ with $a \le 0 \le b$ and $d(\varphi^t(\bar{x}),\varphi^t(\bar{y})) < \epsilon$, there exists $C(\epsilon) > 0$ such that:
	
	\[ (1-C\norm{X_t})\norm{X_t} \le d(\varphi^t(\bar{x}),\varphi^t(\bar{y})) \le (1+C\norm{X_t}) \norm{X_t} \mbox{ for all } t \in [a,b]\]
\end{lemma}

\begin{proof}
	We begin with a few reductions which we have laid out in previous sections:
	
	\[ \begin{array}{rcl}
	d(\varphi^t(\bar{x}),\varphi^t(\bar{y})) & = & d_G(e,\algexp(-tU)g\algexp(tU)) \\
	& = &\abs{\geomlog(\algexp(-tU)\algexp(X)\algexp(tU))} \\
	& = & \abs{\geomlog(\algexp(X_t))}
	\end{array}\]
	
	Now if $\mc U$ is a sufficiently small neighborhood of $0 \in \mf g$, $F = \geomlog \of \algexp : \mc U \to \mf g$ is a $C^\infty$ map such that $d_0F = \id$. In particular, by considering the second-order Taylor expansion for $F$, we know $F(X) = X + f(X)$, where $\limsup_{X \to 0} \norm{f(X)}/\norm{X}^2 = C' < \infty$. Hence by choosing $\epsilon$ sufficiently small and any $C > C'$, we get that $\norm{f(X)} < C\norm{X}^2$. This is exactly the desired result.
\end{proof}

\begin{proposition}
	\label{prop:Bowen-ball}
	Let $\varphi_t$ be a quasi-unipotent flow, and $\bar{x} \in \Gamma \backslash G$. There exists $\epsilon_0 > 0$ such that if $\epsilon < \min\set{\epsilon_0,\inj(\varphi^t(\bar{x}))}$ for all $t \in [0,T]$, there exists $C_0(\epsilon,\varphi) > 0$, ${C_0}^{-1}T^{-R} \le \mu(B_\varphi^T(\bar{x},\epsilon)) \le C_0 T^{-R}$, where $R$ is as in Theorem \ref{thm:main}. Furthermore, the lower bound holds without the assumption on $\inj(\varphi^t(\bar{x}))$.
\end{proposition}

\begin{proof}
	
Without loss of generality, we assume $\bar{x} = e$. Let $C$ be the constant of Lemma \ref{lem:bilipischitz}. If we choose $\epsilon$ sufficiently small, we can ensure $C\norm{X} < 1/2$ whenever $\algexp(X) \in B(e,\epsilon)$. In particular, we get that if $\bar{y} \in B_\varphi^T(e,\epsilon)$, then $\norm{X_t} < 2\epsilon$. Now since $\norm{\cdot}_\infty$ and $\norm{\cdot}$ are equivalent, we know that there exists $C' > 0$ (independent of $T$ and $\epsilon$) such that $\norm{X_t}_\infty < 2C'\epsilon$. This means that $\abs{a_j^i(t)} < 2C'\epsilon$ for pair $i,j$ corresponding to chains and $\abs{b_j^i(t)^2 + c_j^i(t)^2} < \epsilon$ for every pair $i,j$ correspndong to a double chain. Consider the polynomials $a_0^i(t)$. By Lemma \ref{lem:polynomial-growth}, $a_0^i(t) = \sum_{k=0}^{m_i} t^k a_k^i(0)/k!$. Since this polynomial is uniformly less than $2C'\epsilon$, we get that $\abs{a_k^i(0)/k!} < 2C'C(d)T^{-k}\epsilon$. Similarly, for the twisted chains, we get that $\abs{b_k^i(0)/k!}, \abs{c_k^i(0)/k!} < 2C'E(d)T^{-k}\epsilon$. In particular, $X_0$ is contained in the product of the intervals of length $4C'C(m)T^{-k}\epsilon$ on the coordinate $X_k^i$, where $m$ is the maximum of the $m_i$, and the corresponding balls in double cahins. Since for a fixed $i$, these run from $k=0,\dots,m_i$, we conclude that the total volume of the hypercube in $\mf g$ has measure $(4C'C(m)\epsilon)^{\sum m_i}T^{-R}$. This is contained in the fixed compact set $\geomlog(B(e,4C'C(m)\epsilon))$. We can thus bound the Jacobian of $\algexp$ by a constant $M$. In particular:
	
	\[ \mu(B_\varphi^T(e,\epsilon)) \le M(4C'C(m)\epsilon)^{\sum m_i}T^{-R} \]
	
	To get the opposite inequality, one observes that all of the inequalities can be reversed, so that $\alglog(B_\varphi^T(e,\epsilon))$ contains a product of intervals and balls with length and $4(C')^{-1}\max\set{C(m),E(m)}^{-1}T^{-k}\epsilon$. Then again, since Haar in the group is distorted by the Jacobian of $\alglog$, which is bounded near the identity, we get the opposite inequality.
	
	\end{proof}

\begin{proof}[Proof of Theorem \ref{thm:main} - Topological Case, Cocompact $\Gamma$]
	Since we assume $\Gamma$ is cocompact, we $\inj(\Gamma \backslash G) > 0$, and for sufficiently small $\epsilon > 0$ we may apply Proposition \ref{prop:Bowen-ball} for all $\bar{x}$ and $T$. We first show the lower bound. Then we require $C_0^{-1}T^R$ Bowen balls to cover $\Gamma \backslash G$ by Proposition \ref{prop:Bowen-ball}. If one chooses a rate $a(T) = T^{R - \tau}$ for some $\tau > 0$, then the first limit of  \eqref{eq:topo-limit} becomes bounded below by $C_0^{-1}T^{\tau}$ as $T \to \infty$. For $\tau > 0$, this is $\infty$, so the polynomial topological slow entropy is at least $T^R$
	
	For the upper bound, we instead consider the second limit in equation \eqref{eq:topo-limit}. If $\set{x_i}$ is $\epsilon$-separated, then the sets $B_\varphi^T(x_i,\epsilon/2)$ are all disjoint. Since $\mu(B_\varphi^T(\epsilon/2,\bar{x})) \ge {C_0}^{-1} T^{-R}$, we know that we cannot fit more that $C_0T^R$ such disjoint balls. Now if one chooses rate $T^{R+\tau}$, one sees that the second limit of equation \eqref{eq:topo-limit}is $0$. In particular, the polynomial topological slow entropy is at most $T^R$.
\end{proof}

\section{The Slow Variational Principle}\label{sec:SlowVar}

In this section we will prove some relationships between the metric and topological slow entropies of topological dynamical systems.

\subsection{The Variational Principal of Quasi-Unipotent Flows}
Thoughout this section we assume that $\Gamma \backslash G$ is compact. See Section \ref{sec:noncompact} for a treatment of flows on noncompact homogeneous spaces.\footnote{The arguments for the noncompact spaces also apply to the compact ones, but the argument for the compact spaces is significantly simpler and the proof in the noncompact case uses similar ideas, so we include both.} Say that a metric space $X$ is {\it well-partitionable} if for any Borel probability measure $\mu$, compact set $K$ and $\epsilon > 0$, there exists a finite partition $\mc P$ of $K$ whose atoms have diameter less than $\epsilon$ and such that $\mu\left(\bigcup_{\xi \in \mc P} \partial_\epsilon\xi\right) < \epsilon$, where \[\partial_\epsilon\xi = \set{ y \in X : B(y,\epsilon) \cap \xi \not= \emptyset \mbox{ but } B(y,\epsilon) \not\subset \xi}.\] Note that any smooth manifold is well-partitionable. We begin by recalling a corollary of Proposition 2 from \cite{Kat-Tho}\footnote{In \cite{Kat-Tho} the authors consider compact space $X$ but their proof of Proposition 2 generalizes easily to the case where $X$ is  well-partitionable.}:

\begin{theorem}[Slow Goodwyn's Theorem]\label{thm:Goodwyn}
Suppose that $f$ is a continuous, invertible flow or transformation on a well-partitionable $X$. Then for any invariant measure $\mu$ and family of scales $a_\chi$:

  \[ h_{\mu,a_\chi}(f) \le h_{\operatorname{top},a_\chi}(f) \]
\end{theorem}

We now compute the metric slow entropy for the flows considered in the previous sections. We first prove several lemmas. 
We will use the following classical result on polynomials, which is a form of the Chebyshev inequality, see for instance \cite{Brudnyi}.

\begin{lemma}[Brudnyi-Ganzburg inequality]
\label{lem:brudnyi}
Let $V \subset \R$ be an interval, and $\omega \subset V$ a measurable subset. Then for any polynomial $p$ of degree at most $k$:

\[ \sup_V \abs{p} \le \left(\dfrac{4\abs{V}}{\abs{\omega}}\right)^k \sup_\omega \abs{p}\]
\end{lemma}

\begin{lemma}\label{lemma:locbeh}
Let $\varphi_t$ be a quasi-unipotent flow on $\Gamma \backslash G$. Then there exists $1/20>c_\varphi>0$ such that if $\inj(\Gamma \backslash G) > \eta > 0$, $d(\bar{x},\bar{y}) < \eta$ and $S=S_\eta(\bar{x},\bar{y})>0$ is the smallest number for which $d_{\Gamma\backslash G}(\varphi^S\bar{x},\varphi^S\bar{y})=\eta$. Then
$$\left|\{t\in[0,S]: d_{\Gamma\backslash G}(\varphi^t\bar{x},\varphi^t\bar{y})<c_\varphi\eta\}\right|<\frac{S}{10}.$$

\end{lemma}
\begin{proof}
Let $\eta_\varphi$ be such that the assumptions of Lemma \ref{lem:bilipischitz} are satisfied for $\epsilon=\eta_\varphi$. By Lemma \ref{lem:bilipischitz}, we know for $t\in[0,S]$
$$\|X_t\|(1-C\|X_t\|)\leq d_{\Gamma\backslash G}(\varphi^{t}\bar{x},\varphi^{t}\bar{y})\leq\|X_t\|(1+C\|X_t\|).$$

Let $\omega:=\{t\in[0,S] : d_{\Gamma\backslash G}(\varphi^{t}\bar{x},\varphi^{t}\bar{y})< c_{\varphi}\eta\}$ and $V= [0,S]$. We wish to show that $\abs{\omega} < S/10$. Then from above inequality, by choosing $\eta_{\phi}$ small enough, we can guarantee that $C\|X_t\|<\frac{1}{2}$ in Lemma \ref{lem:bilipischitz}, we have $$\|X_t\|<2c_{\varphi}\eta\text{ for }t\in\omega,$$$$\|X_{S}\|\geq\frac{1}{2}\eta.$$

Since $\|X_t\|^2$ is a polynomial in chain basis, then by Lemma \ref{lem:brudnyi}, if $k = \max\set{m_i}$, we have
\begin{equation}
\sup_{V}\|X_t\|^2\leq (\frac{4|V|}{|\omega|})^{k}\sup_{\omega}\|X_t\|^2.
\end{equation}
By definition, it is clear that $\sup_{\omega}\|X_t\|^2\leq 4c_{\varphi}^2\eta^2$ and $\sup_{V}\|X_t\|^2\geq\frac{1}{4}\eta^2$, then
$$\frac{1}{4}\eta^2\leq \left(\frac{4S}{|\omega|}\right)^kc_{\varphi}^2\eta^2.$$
This equivalents to
$$|\omega|\leq 4(2c_{\varphi})^{\frac{2}{k}}S.$$

By picking $c_{\varphi}$ such that $4(2c_{\varphi})^{\frac{2}{k}}<\frac{1}{10}$, we know $|\omega|<\frac{S}{10}$.

\end{proof}

Let us now define the following partition of $\Gamma \backslash G$. Let $\eta < \inj(\Gamma \backslash G)$ and $P_\epsilon$ be a finite partition of $\Gamma \backslash G$ into with atoms of diameter $\leq \epsilon$.  Throughout this section, we fix such a partition $\mathcal{P}:=P_{c_\varphi\eta}$. Note that since the paritions converge to the full $\sigma$-algebra in the limit, the slow entropy is $R$ for sufficiently small $\eta$ suffices.\footnote{If we assumed the flow was ergodic, we could instead without loss of generality assume the partition is generating} 

\begin{lemma}\label{lemma:longblock}
There exists $\epsilon_0,T_0>0$ such that for every $\epsilon<\epsilon_0$, 
 $T>T_0$, if $\bar{d\,}_{\varphi,\mc P}^T(\bar{x},\bar{y})<\epsilon$, then there exists an interval $I\subset[0,T]$ with $|I|\geq 8T/10$ and $d(\varphi^t(\bar{x}),\varphi^t(\bar{y}))<\eta$ for every $t\in I$.
\end{lemma}

Note as an immediate Corollary, we get:

\begin{corollary}
\label{cor:ham-in-bowen}
There exists $\epsilon_0,T_0>0$ such that for every $\epsilon<\epsilon_0$, 
 $T>T_0$, $B_{\varphi,\mc P}^T(\bar{x},\epsilon) \subset \varphi^{-2T/10}(B_\varphi^{6T/10}(\varphi^{2T/10}\bar{x},\epsilon))$
\end{corollary}

\begin{proof}
Let $\epsilon<\epsilon_\varphi$, $T>T_\varphi$ and $\bar{y} \in B_{\varphi,\mc P}^T(\bar{x},\epsilon)$. Then by Lemma \ref{lemma:longblock} there exists $I=[a,b]$, $b-a\geq \frac{8}{10}T$ such that for $t\in[0,\frac{8}{10}T]\subset [0,b-a]$
$$
d(\varphi^{t}(\varphi^a(\bar{x})),\varphi^{t}(\varphi^a(\bar{y})))<c_\varphi\eta.
$$

By $b-a\geq \frac{8}{10}T$, we know $a\leq\frac{2}{10}T$ and $b\geq\frac{8}{10}T$. This is exactly the containment claimed
\end{proof}

\begin{proof}[Proof of Lemma \ref{lemma:longblock}]
Fix $\bar{x},\bar{y}\in \Gamma \backslash G$ such that $\bar{d\,}_{\varphi,\mc P}^T(\bar{x},\bar{y})<\epsilon$ and $\epsilon<\frac{1}{10}$. To simplify notation we write $\bar{x}_t=\varphi^t\bar{x}, \bar{y}_t=\varphi^t\bar{y}$.
Divide the interval $[0,T]$ by the following procedure: let $t_0\geq 0$ be the smallest for which $d(\bar{x}_{t_0},\bar{y}_{t_0}) \le c_\varphi\eta$. Let $S_0=S(\bar{x}_{t_0},\bar{y}_{t_0})$ be given by Lemma \ref{lemma:locbeh}. Now inductively we define $t_{i+1}\geq t_i+S_i$
to be the smallest for which $d(\bar{x}_{t_i},\bar{y}_{t_i}) \le c_\varphi\eta$. We continue until $t_{J+1}\geq T$. Hence by the definition of $(t_j)$ and $(S_j)$ it follows that
\begin{equation}\label{eq:sj}
d(\bar{x}_t,\bar{y}_t)\leq \eta \text{ for every } t\in[t_{J},\tilde{T}],
\end{equation}
where $\tilde{T}:=\min(t_J+S_J,T)$. Moreover, since $t_{J+1}\geq T$ it follows that
\begin{equation}\label{eq:sj2}
d(\bar{x}_t,\bar{y}_t)\geq c_\varphi \eta \text{ for every } t\in[\tilde{T},T],
\end{equation}
 By Lemma \ref{lemma:locbeh}, for every $i=0,...,J$
$$
\left|\{t\in [t_i,t_i+S_i]: d(\varphi^t(\bar{x}),\varphi^t(\bar{y}))\geq c_\varphi \eta\}\right|\geq \frac{9}{10}S_i.
$$
Using \eqref{eq:sj2} and the fact that $\bar{d\,}_{\varphi,\mc P}^T(\bar{x},\bar{y})<\epsilon$, we get
\begin{multline*}
(1-\epsilon)T\leq |\{t\in[0,T]: d(\bar{x}_t,\bar{y}_t)<c_\varphi\eta\}|=|\{t\in[0,\tilde{T}]: d(\bar{x}_t,\bar{y}_t)<c_\varphi\eta\}|=\\
|\{t\in[0,t_J]: d(\bar{x}_t,\bar{y}_t)<c_\varphi\eta\}|\}|+|\{t\in[t_{J},\tilde{T}]: d(\bar{x}_t,\bar{y}_t)<c_\varphi\eta\}|\leq \frac{t_{J}}{10}+(\tilde{T}-t_{J}).
\end{multline*}
Hence (since $t_J\leq T$), $|\tilde{T}-t_{J}|\geq (1-\epsilon)T-\frac{T}{10}\geq \frac{8T}{10}$.
Defining $I:=[t_{J},\tilde{T}]$ finishes the proof by \eqref{eq:sj}.
\end{proof}

Using Lemma \ref{lemma:longblock} we can prove the metric part of Theorem \ref{thm:main}.

\begin{proof}[Proof of Theorem \ref{thm:main} (cocompact lattice metric part)]


By Corollary \ref{cor:ham-in-bowen}, $\mu(B_{\varphi,\mc P}^T(\bar{x},\epsilon)) \le \mu( \varphi^{-2T/10}(B_\varphi^{6T/10}(\varphi^{2T/10}\bar{x},\epsilon)))$. Notice that by Proposition \ref{prop:Bowen-ball}
$$\mu(B_{\varphi}^{\frac{6}{10}T}(\varphi^{\frac{2}{10}T}\bar{x},c_{\varphi}\eta))\leq C_0(c_{\varphi}\eta,\varphi)T^{-R},
$$
hence the metric growth is large equal than $T^{R}$. By Theorem \ref{thm:Goodwyn} and topological cocompact case of Theorem \ref{thm:main}, the metric growth is asymptotically equal to $T^R$.
\end{proof}

\subsection{Characterization of Kronecker Systems}

We now show the failure of the variational principal in a general setting. We do so even in the smooth category. We first show some certain characterizations of transformations with zero entropy at every scale $a$. For metric slow entropy, this was fist done in \cite{Ferenczi} using different language:

\begin{proposition}[\cite{Ferenczi}]
A measure-preserving transformation or flow $f$ has $h_{\mu,a_\chi}(f) = 0$ with respect to every family of scales $a_\chi$ if and only if $f$ is measurably conjugate to an action by translations on a compact abelian group
\end{proposition}

We prove an analogy in the topological category, namely:

\begin{proposition}
\label{prop:zero-all-scales}
A minimal homeomorphism $f : X \to X$ of a compact metric space has $h_{\operatorname{top},a_\chi}(f) = 0$ for every family of scales $a_\chi$ if and only if $f$ is topologically conjugate to a translation on a compact abelian group.
\end{proposition}

We use the following folklore characterization:

\begin{lemma}
\label{lem:equi}
Let $f : X \to X$ be a transitive homeomorphism or flow of a compact metric space. Then $f$ is topologically conjugate to translations on a compact abelian group if and only if $\set{f^t}$ is uniformly equicontinuous.
\end{lemma}

\begin{proof}[Proof of Proposition \ref{prop:zero-all-scales}]
First, we show that if $f$ is topologically conjugate to a translation on a compact abelian group, it has zero topological slow entropy at all scales. If $f$ is topologically conjugate to a translation of a compact abelian group, we are free to use the bi-invariant metric for which $f$ is an isometry. Hence, the Bowen balls are equal to the standard balls for arbitrary $n$. In particular, $N_{f,X}(\epsilon,T)$ is independent of $T$ for any $\epsilon$. Since $a(T)$ must tend to $\infty$, we get the result.

We now prove the converse. That is, we show that if we have zero entropy at all scales, then $f$ is topologically conjugate to a translation on a compact abelian group. Since $f$ is assumed to have 0 entropy at all scales, $N_{f,X}(\epsilon,T) \le N_0$ for some fixed $N_0 = N_0(\epsilon)$. Let $Y_t(\epsilon) \subset X^{N_0}$ be the set of $N_0$-tuples such that $(x_i) \in Y_t$ if and only if $\bigcup_i B_f^t(x_i,\epsilon) = X$. Then $Y_t(\epsilon)$ is nested and nonempty for every $t$, so $\bigcap_{t >0} \overline{Y_t(\epsilon/2)} \not= \emptyset$.

We claim that $\overline{Y_t(\epsilon/2)} \subset Y_t(\epsilon)$. Indeed, fix $t$ and suppose that $(x_i^k) \in Y_t(\epsilon/2)$ is a sequence of $N_0$-tuples (indexed by $k$) converging to $(x_i)$. $(x_i^k) \in Y_t(\epsilon/2)$ is equivalent to the property that if $y \in X$, there exists some $i$ such that $d(f^{t'}(x_i^k),f^{t'}(y)) < \epsilon/2$ for every $0 \le t' \le t$. Since $d$ and $f^{t'}$ is continuous for each $t'$, $d(f^{t'}(x_i), f^{t'}(y)) \le \epsilon/2 < \epsilon$. This proves the claim.

We may thus find an $N_0$-tuple $(x_i)$ so that $\bigcup_i \overline{B_f^t(x_i,\epsilon)} = X$ for every $t$. Since the sets are nested, and each point must lie in such a neighborhood for every $t$, the sets $B_f^\infty(x_i,\epsilon) = \bigcap_{t > 0} \overline{B_f^t(x_i,\epsilon)}$ still cover $X$. They are closed sets, but since finitely many cover the space $X$, at least one must have nonempty interior. That is, for some $z \in Z$ and $0 < \gamma < \epsilon/2$, $B(z,\gamma) \subset \overline{B_f^t(x_i,\epsilon)} \subset B_f^t(x_i,2\epsilon)$ for every $t > 0$.

Since $f$ is assumed to be minimal, we may find $T$ such that $\bigcup_{t= 0}^T f^{-t}(B(z,\gamma)) = X$. Equivalently, for every $x \in X$, there exists some $t \le T$ such that $f^t(x) \in B(z,\gamma)$. Since any compact family of homeomorphisms of a compact space is unifomly equicontinuous, we may choose $\delta$ so that $f^t(B(x,\delta)) \subset B(f^t(x),\gamma)$ for $t \le T$.

Now, given any $x \in X$, by choice of $T$, we may find $t \le T$ such that $d(f^t(x),z) < \gamma$. Furthermore, by choice of $\delta$, we conclude that $f^t(B(x,\delta)) \subset B(f^t(x),\gamma) \subset B(z,2\gamma) \subset B_f^s(x_i,\epsilon)$ for every $s > 0$. Then $f^{t+s}(B(x,\delta)) \subset B(f^s(x_i),\epsilon)$, and if $d(x,y) < \delta$, $d(f^t(x),f^t(y)) < 2\epsilon$ for every $t \ge 0$. Since $\epsilon$ was arbitrary, we conclude that $\set{f^t}$ is uniformly equicontinuous, and the result.
\end{proof}

\subsection{Failure of the Variational Principal}\label{sec: counterexampleVar}

If $f : X \to X$ is a homeomorphism of a compact metric space, let $\mc M(f)$ denote the space of invariant measures. The results of the previous section show that if we can find a minimal topological system $f : X \to X$ such that $\mc M(f)$ is finite dimensional and every $\mu \in \mc M(f)$ is Kronecker\footnote{It is not enough to say that every $\mu \in \mc M(f)$ is Kronecker, since $(x,y) \mapsto (x,y+x)$ has this property but nontrivial slow entropy}, but which is not topologically conjugate to a translation on a compact abelian group. This is possible by non-standard realization theory, and the {\it approximation-by-conjugation method} first used by Anosov and Katok in \cite{AnosovKatok}. We document the conclusion here:

\begin{proposition}
Let $M$ be a manifold with a free circle action. Then there exists a uniquely ergodic, volume preserving, $C^\infty$ diffeomorphism $f : M \to M$ which is measurably conjugate to a translation on a torus $\T^d$, $d \ge 1$.
\end{proposition}

As a result we get the following

\begin{corollary}
Let $M$ be a manifold with a free circle action. Then there exists a $C^\infty$ diffeomorphism $f : M \to M$ and family of scales $\set{a_\chi}$ such that:

\[ \sup_{\mu \in \mc M(f)} h_{\mu,a_\chi}(f) < h_{\operatorname{top},a_\chi}(f) \]
\end{corollary}

A similar example can be found by using an example due to Furstenburg \cite{Fur}, found in \cite{Katok-Has}:

\begin{proposition}
There exists a minimal $C^\infty$ diffeomorphism $f : \T^2 \to \T^2$ which is measurably conjugate to $(x,y) \mapsto (x+\alpha,y)$ for some $\alpha \in S^1$. In particular,

\[ \sup_{\mu \in \mc M(f)} h_{\mu,a_\chi}(f) < h_{\operatorname{top},a_\chi}(f) \]
\end{proposition}

Another interesting class we note are the {\it Sturmian sequences}. These are zero entropy closed subshifts on two symbols which are uniquely ergodic, and whose unique measure gives a transformation measurably isomorphic to a circle rotation. These can be obtained in the smooth category by taking the non-wandering set of the $C^1$ Denjoy examples of a nontransitive diffeomorphism. Again, because the only translations on compact abelian groups which are Cantor sets are the odometers, we get the following result:

\begin{proposition}
Any Sturmian subshift $\sigma$ preserving an invariant measure $\mu$ satisfies

\[ 0 = h_{\mu,a_\chi}(\sigma) < h_{\operatorname{top},a_\chi}(\sigma) \]

for some family of scales $\set{a_\chi}$
\end{proposition}

These examples motivate the following questions:

\begin{question}
For which continuous transformations $T : X \to X$ do we have $\sup_{\mu \in \mc M(T)}h_{\mu,a_\chi} = h_{\operatorname{top},a_\chi}$ for every family of scales $a_\chi$?
\end{question}

Another observation on these examples is the following: for a given Kronecker system, one may find an ``ideal'' realization which links the measurable orbit growth structure with the topological orbit growth structure. In particular, we ask the following question:

\begin{question}
Let  $T : (X,\mu) \to (X,\mu)$ be a measure-preserving transformation of a probability space $(X,\mu)$. For which $T$ does there exist a (unique) topological system $S : Y \to Y$ and a measurable map $H : X \to Y$ such that $H_*(\mu)$ is fully supported, $H \of T = S \of H$ and $h_{\mu,a_\chi}(T) = h_{\operatorname{top},a_\chi}(S)$ for every family of scales $a_\chi$?
\end{question}

\section{Sequence Entropy of Quasi-Unipotent Flows}\label{sec:SeqeunceEntropy}
Recall the definitions of Section \ref{sec:top-seq-entropy}. We prove the following corollary of Proposition \ref{prop:Bowen-ball}.
\begin{corollary}\label{SequenceBowen}
Fix  $L > 0$, $\lambda > 1$ and $\epsilon > 0$ sufficiently small. There exists $C_0(\epsilon,L,\lambda,\varphi) > 0$ such that for every $\epsilon$-good $\bar{x} \in \Gamma \backslash G$, $C_0^{-1}\lambda^{-NR} \le \mu(B_\varphi^N(\bar{x},\epsilon;\set{L\lambda^k})) \le C_0\lambda^{-NR} $.
\end{corollary}

\begin{proof}
Since $B_\varphi^{L\lambda^N}(\bar{x},\epsilon) \subset B_\varphi^N(\bar{x},\epsilon;\set{L\lambda^k})$, the lower bound follows from Proposition \ref{prop:Bowen-ball}. The upper will follow if we show that $B_\varphi^N(\bar{x},\delta;\set{L\lambda^k})\subset B_\varphi^{L\lambda^N}(\bar{x},\epsilon)$ for some $\delta,\epsilon >0$ independent of $N$. We prove the inclusion inductively. For the base step, notice that for any $\epsilon$ we may choose $\delta > 0$ so that this holds simply by continuity of $\varphi^L$. For the inductive step, we suppose we have it for $N$. Then from the proof of Proposition \ref{prop:Bowen-ball}, we know in fact that if $x$ is a lift of $\bar{x}$, and $\bar{y} \in B_\varphi^{L\lambda^N}(\bar{x},\epsilon)$, using the notation in the proof of Proposition \ref{prop:Bowen-ball}, $\abs{a_k^i(0)} < C_1\lambda^{-NR}\epsilon$. This implies that $\abs{a_k^i(L\lambda^{N+1})} < C_2(L\lambda^{N+1})^{R}C_1\lambda^{-NR}\epsilon = C_1C_2L^{R}\lambda^R\epsilon$ and thus $\|X_t\|<C_3\epsilon$ for $0\leq t\leq L\lambda^{N+1}$. But since distances are distorted up to a uniform constant by taking algebraic and geometric exponentials, we may choose $\epsilon$ to have the following property: if $\bar{y} \in B_\varphi^N(\bar{x},\delta;\set{L\lambda^k})$, $d(\varphi^{L\lambda^{N+1}}(\bar{x}),\varphi^{L\lambda^{N+1}}(\bar{y})) = d(\varphi^{L\lambda^{N+1}}(x),\varphi^{L\lambda^{N+1}}(y))$, where $x$ and $y$ are lifts chosen to minimize the distance at time $t =0$. In particular, the points never split on the universal cover, and we know they are Bowen close up to time $L\lambda^{N+1}$.

\end{proof}

\begin{proof}[Proof of Theorem \ref{thm:mainco}]
With Corollary \ref{SequenceBowen} in hand, the computation for topological sequence entropy follows from an argument virtually identical to that of the proof of Theorem \ref{thm:main}. We give a brief argument on the measure-theoretic case. In order to simplify the notation, we denote $T=\varphi_L$.

It is sufficient to consider a family of partitions $\{\mc P_\alpha\}$ which atoms are small cube under the chain structure whose radius go to zero as $\alpha \to\infty$. 
We can pick $\alpha$ large enough to guarantee the diameter of each atom is less than $\epsilon$, we write $\mc P = \mc P_\alpha$ to simplify notation. Then consider the measure of atom of $\mc P^{-n}=\bigvee_{k=0}^n T^{-L\lambda^k}\mc P$. Notice that if two points $\bar{x},\bar{y}$ belong to the same atom of $\mc P^{-n}$, they stay $\epsilon$ close for iterations $L,L\lambda,\ldots,L\lambda^n$ of $T$, which is equivalent to $y$ belongs to $B_\varphi^n(\bar{x},2\epsilon;\set{L2^k})$.  Thus the maximal of measure of $\mc P^{-n}$'s atom should no large than the supremem of $\mu(B_\varphi^n(\bar{x},2\epsilon;\set{L2^k}))$ over all $\bar{x}$. Then by Corollary \ref{SequenceBowen}, we have
 $$\mu(B_\varphi^n(\bar{x},2\epsilon;\set{L\lambda^k}))\le C_0^{-1}\lambda^{-nR}.$$

This immediately yields that $h_{A,\mu}(\varphi) \ge R \log \lambda$. For the upper bound, by the same method in standard Variational principle, we have $h_{A,\mu}(\varphi) \le h_{A,\operatorname{top}}(\varphi) = R \log \lambda$.
\end{proof}

\section{Slow entropy of flows on non-compact spaces}
\label{sec:noncompact}
Although in previous sections we only deal with the cocompact lattice, our method in fact can be applied to non-cocompact lattice case by following arguments.

\subsection{Hamming Balls estimates in Noncompact Homogeneous Spaces}
In this section, we suppose $\Gamma\backslash G$ is not compact. Let $K\subset \Gamma\backslash G$ be such that $\mu(K)>1-\delta$ for some $\delta>0$. Given $\epsilon > 0$, choose a partition $\mc P_{\epsilon}$ of $\Gamma\backslash G$ such that $K^c$ is one atom of partition and the diameter of the remaining atoms is less than $\epsilon$.

\begin{proposition}\label{prop:NoncompactProp}
Let $K$ be a compact subset of $\Gamma \backslash G$ such that $\mu(K)>1-\delta$ for some $\frac{1}{100}>\delta>0$ and $\eta = \inj(K)$. Choose a partition $\mc P := {\mc P}_{c_\varphi\eta}$ as above, where $0 < c_\varphi < 1$ is as in Lemma \ref{lemma:locbeh} applied to the flow on the universal cover. Then there exists $T_0,\epsilon_0 > 0$, a set $L \subset K$ such that $\mu(L) > 1 - 2\delta$ with the following property: for $0 < \epsilon < \epsilon_0$, $\bar{x} \in L$ and $\bar{y} \not\in \bar{x}C_G(U)$, if $\bar{d\,}_{\varphi,\mc P}^T(\bar{x},\bar{y}) < \epsilon$, then $d_\Phi^{\frac{7}{10}T}(x_{\frac{3}{10}T},y_{\frac{3}{10}T}) < \eta$, where $x_{\frac{3}{10}T},y_{\frac{3}{10}T}$ are lifts of 
$\varphi^{\frac{3}{10}T}\bar{x},\varphi^{\frac{3}{10}T}\bar{y}$ minimizing the distance in the universal cover.
\end{proposition}
\begin{proof}
Let $L_0(T_0,\xi) \subset X$ denote the set of all points $\bar{x} \in \Gamma \backslash G$ such that \[\set{t \in [0,T] : \varphi^t(\bar{x}) \in K} \ge (1-\xi)\mu(K)T \] for all $T \ge T_0$. Then by the ergodic theorem, for any $\xi > 0$, we may find $T_0$ sufficiently large such that $\mu(L(T_0,\xi)) > 1-\delta$. We will specify $\eta$ later, and take $L = L_0 \cap K$, so that $\mu(L) \ge 1-2\delta$.

Let $T \ge T_0$, $\bar{x} \in L$,
 $\bar{y}$ be such that $\bar{d\,}_{\varphi,\mc P}^T(\bar{x},\bar{y}) < \epsilon$ and $x_t,y_t \in G$ be lifts of $\varphi_t(\bar{x}),\varphi_t(\bar{y})$ to $G$ such that $d_G(x_t,y_t) = d_{\Gamma\backslash G}(\varphi_t(\bar{x}),\varphi_t(\bar{y}))$. Let $\Phi_t$ denote the lift of $\varphi_t$ to the universal cover $G$. Divide the interval $[0,T]$ into subintervals by following method. We inductively define a sequence $S_i$ and $T_i$ as follows. Let $T_0 = 0$, and

\begin{eqnarray*}
S_i & = & \min\{t \ge T_i : \varphi^t\bar{x}\in K\text{ and }d_{\Gamma\backslash G}(\varphi^t\bar{x},\varphi^t\bar{y})\leq c_{\varphi}\eta\} \\
T_{i+1} & = & \min\{t \ge S_i : d_G(\Phi^{t-S_i}(x_{S_i}),\Phi^{t-S_i}(y_{S_i}))=\eta\}
\end{eqnarray*}

Note that $S_0 < \infty$ since $\bar{x} \in L$. 
Then by Lemma \ref{lemma:locbeh} for universal cover, we have the following claim.
\begin{claim}\label{claim:longBadTime}
For any $i\geq1$ such that $T_i\leq T$, we may choose $C_1>0$ (as in the definition of $\epsilon_0$) independent of $i$ such that for set $B_i=\{t\in[S_{i-1},T_i] : d_G(\Phi^{t-S_{i-1}}x_{S_{i-1}},\Phi^{t-S_{i-1}}y_{S_{i-1}})\geq c_{\varphi}\eta\}$, where $x_{S_{i-1}},y_{S_{i-1}}$ are the lifts of $\varphi^{S_{i-1}}\bar{x},\varphi^{S_{i-1}}\bar{y}$ minimize the distance at time $S_{i-1}$ in the universal cover, then we have
$$|B|\geq\frac{9}{10}(T_i-S_{i-1}).$$
\end{claim}







Now denote $A=\{t\in[0,T] : \varphi^t\bar{x}\in K\text{ and }d_{\Gamma\backslash G}(\varphi^t\bar{x},\varphi^t\bar{y})<\frac{c_{\varphi}\eta}{2}\}$. Recall that the atoms of ${\mc P}_{c_\varphi\eta}$ have diameter less than $c_\varphi\eta$ other than $K^c$. Therefore, by choosing $\epsilon_0$ (which bounds $\bar{d\,}_{\varphi,\mc P}^T(\bar{x},\bar{y})$) and $\xi$ (as in the definition of $L$) small enough, we can guarantee $|A|\geq\frac{9}{10}T$. 
If $t\in(T_i,S_i)$, either $\varphi^t\bar{x}\notin K$ or $d_{\Gamma\backslash G}(\varphi^t\bar{x},\varphi^t\bar{y})>c_{\varphi}\eta$. In particular, if $S_i < T$

\begin{equation}\label{eq:ComplentaryBadSet}
(T_i,S_i)\cap A=\emptyset,
\end{equation}
and thus
$$\abs{\bigcup(T_i,S_i)}<\frac{1}{10}T.$$


From the Claim \ref{claim:longBadTime}, we know that whenever $T_i < T$
\begin{equation}\label{eq:longB_i}
|\bigcup_iB_i|\geq \frac{9}{10}|\bigcup_i [S_{i-1},T_i]|.
\end{equation}

We claim that $A \cap B_i = \emptyset$ for every $i$. Indeed, if $t\in A \cap B_i$, then \[c_{\varphi}\eta\leq d_G(\Phi^{t-S_i}x_{S_i},\Phi^{t-S_i}y_{S_i})\leq\eta = \inj(K).\] Therefore, $d_{\Gamma\backslash G}(\varphi^t\bar{x},\varphi^t\bar{y})\geq c_{\varphi}\eta$ and thus contradicts to the definition of $t\in A$. 
Then we have following inequality,
\begin{equation}
\begin{aligned}
|A|\leq |A\cap(\bigcup B_i)|+|A\cap(\bigcup ([S_{i-1},T_i]\setminus B_i))|+|A\cap(\bigcup (T_i,S_i))|+|A\cap I|.
\end{aligned}
\end{equation}

where $I$ is either $[T_{i_0},T]$ or $[S_{i_0},T]$. Recall that the first and third term on the right equal to $0$ (due to \eqref{eq:ComplentaryBadSet} and paragraph after \eqref{eq:longB_i}) and second term is less than $\frac{1}{10}T$(due to \eqref{eq:longB_i}), thus we know $|A\cap I |\geq\frac{7}{10}T$. In particular, $I$ takes the form $[S_{i_0},T]$ (since $A \cap (T_i,S_i) = \emptyset$ for every $i$) and has length at least $7/10$.
By setting $[a,b] = [S_{i_0},T]$, we conclude the proposition.
\end{proof}

\subsection{Proof of Theorem \ref{thm:main} - Noncompact Case}
We show that the measure-theoretic polynomial entropy is at least $R$ and the topological polynomial entropy is at most $R$. Then by Theorem \ref{thm:Goodwyn}, we conclude the main theorem. We first show the lower bound. By Proposition \ref{prop:NoncompactProp}, we know that the Hamming ball with radius $\epsilon$, center $\bar{x}$ and partition ${\mc P}$ will be contained in $$\varphi^{-\frac{3}{10}T}\pi\left(\bigcup_{x'\in\pi^{-1}(\varphi^{\frac{3T}{10}}\bar{x})} B_\Phi^{7T/10}(x',\eta)\right).$$

Recall that $\Gamma$ acts on left, quasi-unipotent flow acts on right and our metric is left invariant, thus for any $x_1,x_2\in\pi^{-1}(\varphi^{\frac{3T}{10}}\bar{x})$, we have
$$\pi(B_\Phi^{7T/10}(x_1,\xi))=\pi(B_\Phi^{7T/10}(x_2,\eta)).$$

Hence since $\inj(x) = \infty$ for the flow $\Phi$ on the universal cover, by Proposition \ref{prop:Bowen-ball} applied to $\Phi$:
\[ \mu(B_{\varphi,\mc P}^T(\bar{x},\epsilon/2)) \le C_0(7T/10)^{-R}. \]

Given any cover of a large measure subset of $L$ by Hamming balls, we can replace the centers with points of $L$, which will still cover $L$ if we double the radius $\epsilon/2$ to $\epsilon$. So we need at least $CT^{R}$ $(\epsilon/2,T)$-Hamming balls to cover $L$.

Now we turn to the topological category. Notice that the lower bound of Proposition \ref{prop:Bowen-ball} applies without the assumption on $\inj(\Gamma \backslash G)$. So as in the proof of the compact case, we get the upper bound on slow entropy is $T^R$.


\section{Computation of Slow Entropy for Specific Examples}
\label{sec:specific}
In this section we prove Corollary \ref{cor:sl2-triple} and compute the polynomial entropy of Examples \ref{ex:semisimple}-\ref{ex:twisted} (Corollary \ref{cor:specific}). Since we will use Corollary \ref{cor:sl2-triple} in the computations of the examples, we prove it first.

\subsection{Proof of Corollary \ref{cor:sl2-triple}}

Let $G$, $U$ and $(V,X,U)$ be as in the discussion preceding the statement. The bracket relations on $V$, $X$ and $U$ show that they correspond to an algebra homomorphism $\varphi : \mf{sl}(2,\R) \to \mf g$ defined via:

\[ \varphi\begin{pmatrix} 0 & 0 \\ 1 & 0 \end{pmatrix} = U \qquad \varphi\begin{pmatrix} 1 & 0 \\ 0 & -1 \end{pmatrix} = X \qquad \varphi\begin{pmatrix} 0 & 1 \\ 0 & 0 \end{pmatrix} = V \]

This defines a representation of $\mf{sl}(2,\R)$ by sending $Y \mapsto \ad_Y : \mf g \to \mf g$. Since $\mf{sl}(2,\R)$ is semisimple, this representation can be decomposed into a sum of irreducible ones. The finite-dimensional irreducible representations of $\mf{sl}(2,\R)$ are well-classified, and are indexed by $\N_0$. If $n \in \N_0$, let $V_n$ be the space spanned by $\set{X_0,\dots,X_n}$, so that $\dim(V_n) = n+1$. Then the action of $\mf{sl}(2,\R)$ is given by:

\[ \begin{pmatrix} 0 & 0 \\ 1 & 0 \end{pmatrix} \cdot X_j = c_{n,j}X_{j+1} \qquad \begin{pmatrix} 1 & 0 \\ 0 & -1 \end{pmatrix} \cdot X_j = (2j-n)X_j \qquad
\begin{pmatrix} 0 & 1 \\ 0 & 0 \end{pmatrix} \cdot X_j = c_{n,j}'X_{j-1} \]

for some nonzero universal constants $c_{n,j}$ and $c_{n,j}'$ and we make the convention that $X_{-1} = 0$ and $X_{n+1} = 0$. In particular, each $V_n$ gives a Jordan block of length $n+1$ for $U$. Note that if $[U,Y] = 0$, we may decompose $Y$ via a decomposition into the irreducible subrepresentations. Thus, the centralizer of $U$ is exactly the sum of the centralizers in each irreducible subspace. But only the $X_n$ terms in each irreducible subspace commute with $U$, and here the eigenvalue of $X$ is $n$. So for each eigenvector of $X$ with eigenvalue $n$, there is a corresponding chain of length $n+1$ (which is depth $n$), and this gives a basis. Applying Theorem \ref{thm:main}, we get the Corollary.

\subsection{Calculation for Example \ref{ex:semisimple}}

Let us begin by assuming that we have a single block in $SL(d,\R)$ (in this case the nilpotent element $U$ is sometimes called {\it principal}). According to Corollary \ref{cor:sl2-triple}, we may identify an $\mf{sl}(2,\R)$ triple and the centralizer of $U$. Direct computation shows an $\mf{sl}(2,\R)$ triple can be constructed by taking:

\[ X = \begin{pmatrix}
d-1 \\
& d-3 \\
 & & \ddots \\
 & & & -d+1
\end{pmatrix} \qquad V = \begin{pmatrix}
0 \\
d-1 & 0 \\
 & \ddots & \ddots \\
& & kd -k^2 & 0\\
& & & \ddots & \ddots \\
& & & & & d-1 & 0
\end{pmatrix}\]

Let $E_{i,j}$ denote the matrix whose entries are all zero, except for the $(i,j)$th entry which is 1. Since the eigenvalues of $X$ on the centralizer of $U$ must be positive integers, know that that the centralizer lies completely in the upper triangular matrices and any element must be in the sums of $W_n = \bigoplus_{i=1}^{d-k} \R E_{i,i+k}$ for a fixed $k \ge 0$ (since these are the eigenspaces of $X$). From direct computation one sees that each such subspace, there is a unique element (up to scaling) commuting with $U$ (it is exactly $\sum_{i=1}^{d-k} E_{i,i+k}$). Hence, the eigenvalues of $\ad_X$ on $C(U)$ are all simple, and equal to $2,4,\dots,2d-2$. So by Corollary \ref{cor:sl2-triple}, the topological slow entropy is:

\[ \sum_{k=1}^{d-1} 2k(2k+1)/2 = \frac{d}{6}(d-1)(4d+1)\]

This shows the result for a single block. To consider multiple blocks, note that there exists an $\mf{sl}(2,\R)$ triple respecting the block form of $U$, by taking the block forms of $X$ and $V$. So in each $k_i \times k_i$ block, we may apply our analysis from the previous sections for elements of the centralizer in each block. This accounts for the first sum in the expression. Each off-diagonal block is preserved by the $\mf{sl}(2,\R)$ triple, so we will find a basis of the centralizer by considering each block independently. Furthermore, since the upper triangular blocks are dual to the lower triangular blocks and all $\mf{sl}(2,\R)$ representations are self-dual, we may consider only the upper triangular ones and double the result. Again, an element of the centralizer contributing to the slow entropy must lie in a positive eigenspace of $X$, so if we consider a block of $k_i \times k_j$ matrices (we suppose that $k_j\geq k_i$), elements of the centralizer must lie in $\bigoplus_{a=1}^{k_j-b}\R E_{a,a+b}$ for a fixed $0 \le b \le k_j-1$. Direct computation shows that an element of the centralizer exists in such a subspace (and is unique up to scalar) if and only if $k_j - k_i \le b \le k_j - 1$ (and is given by $\sum_{a = 1}^{k_j-b} E_{a,a+b}$). The eigenvalue of $X$ on such a subspace is given by $\set{k_j-k_i+2\ell : 0 \le \ell \le k_i-1}$.

 Hence each block, indexed by $1 \le i < j \le m$, we have a contribution to slow entropy equal to:

\[ \sum_{\ell=0}^{k_i-1} (k_j-k_i+2\ell)(k_j-k_i+2\ell+1)/2 =  \frac{1}{6} k_i \left(k_i^2+3k_j^2-3k_j-1\right) \]

Recalling that we must double the contribution for the lower-triangular blocks, we get the formula.

\subsection{Calculation for Example \ref{ex:nilpotent}}

In this case it is easy to find a chain basis consisting of two chains, since one may check that ${\ad_U}^d = 0$ but ${\ad_U}^{d-1} \not= 0$. As a result, there is at least one chain of depth $d-1$ (and dimension $d$). But since $\dim(\mf n) = d+1$, the only possiblitiy is to have one chain of length $d$ and another of length one. This gives the formula immediately.

\subsection{Calculation of Example \ref{ex:twisted}}

Observe that because $U \in \mf{sl}(d,\R)$, $\ad_U$ preserves the subspaces $\mf{sl}(d,\R)$ and $\R^N$. We can hence find a chain basis subordinate to this splitting. We may use the calucluations for Example \ref{ex:semisimple} to get the first part. Note that $\ad_U$ acts on $\R^N$ as $d\rho(U)$, a nilpotent matrix, whose Jordan normal form corresponds to a chain basis (see the proof of Lemma \ref{lem:chain-basis}). This gives the formula.

\section*{Acknowledgements}
The authors would like to thank Anatole Katok for warm encouragements and many helpful discussions.

\end{document}